\documentclass[a4paper, 12pt,oneside,reqno]{amsart}
\usepackage[a4paper]{geometry}
\usepackage[english]{babel}
\usepackage[T1]{fontenc}
\usepackage{amsfonts}
\usepackage{mathrsfs}
\usepackage{tikz}
\usetikzlibrary{arrows,shapes,snakes,automata,backgrounds,petri,through,positioning}
\usetikzlibrary{intersections}

\usepackage[matrix,arrow]{xy}
\usepackage{amssymb,amscd,amsthm,amsmath}
\usepackage[a4paper]{geometry}
\usepackage{amsmath}
\usepackage{amssymb}
\usepackage{amsthm}
\usepackage{graphicx}
\usepackage[colorlinks=true, allcolors=blue]{hyperref}

\numberwithin{equation}{section}

\newtheorem{theorem}{Theorem}[section]
\newtheorem{lemma}[theorem]{Lemma}

\newtheorem{proposition}[theorem]{Proposition}

\theoremstyle{definition}

\newtheorem{example}[theorem]{Example}
\newtheorem{remark}[theorem]{Remark}
\newtheorem{problem}{Problem}

\newcommand{\GG}{\mathbb{G}}
\def\sl{\mathfrak{sl}}

\usepackage{array}
\newcolumntype{P}[1]{>{\centering\arraybackslash}p{#1}}
\title{On finite-dimensional homogeneous Lie algebras of derivations of polynomial rings}
\author{Ivan Arzhantsev}
\address{HSE University, Faculty of Computer Science, Pokrovsky Boulvard 11, Moscow, 109028 Russia}
\email{arjantsev@hse.ru}
\author{Sergey Gaifullin}
\address{Lomonosov Moscow State University, Faculty of Mechanics and Mathematics, Department of Higher Algebra, Leninskie Gory 1, Moscow, 119991 Russia;
Moscow Center of Fundamental and Applied Mathematics, Moscow, Russia;  
HSE University, Faculty of Computer Science, Pokrovsky Boulvard 11, Moscow, 109028 Russia}
\email{sgayf@yandex.ru}
\author{Viktor Lopatkin}
\address{HSE University, Faculty of Computer Science, Pokrovsky Boulvard 11, Moscow, 109028 Russia}
\email{vlopatkin@hse.ru}
\thanks{The research was done within the framework of the HSE Fundamental Research Program in 2024} 
\begin{document}
\maketitle
\begin{abstract}
For a finite set of homogeneous locally nilpotent derivations of the algebra of polynomials in several variables, a finite dimensionality criterion for the Lie algebra generated by these derivations is known. Also the structure of the corresponding finite-dimensional Lie algebras is described in previous works. In this paper, we obtain a  finite dimensionality criterion for a Lie algebra generated by a finite set of homogeneous derivations, each of which is not locally nilpotent.
\end{abstract}
\section{Introduction}

Let $\mathbb{K}$ be a field of characteristic zero. Recall that a derivation of a $\mathbb{K}$-algebra $A$ is a linear operator $D\colon A\to A$ satisfying the Leibniz rule
$D(fg)=D(f)g+fD(g)$ for all $f,g\in A$.

Let us consider the algebra of polynomials $\mathbb{K}[x_1,\ldots,x_n]$. It is easy to show that each derivation $D$ of this algebra can be written as 
\[
D=f_1\partial_1+f_2\partial_2+\ldots+f_n\partial_n,
\]
where $\partial_i:=\frac{\partial}{\partial x_i}$ is the partial derivative with respect to the variable $x_i$ and $f_1,f_2,\ldots,f_n$ are arbitrary polynomials.

It is well known that the set of all derivations of the algebra $\mathbb{K}[x_1,\ldots,x_n]$ form a Lie algebra with respect to the commutator $[D_1,D_2]:=D_1\circ D_2-D_2\circ D_1$. This algebra is called the Lie algebra of \textit{Cartan type $W$} (see~\cite{Ca}) and is denoted by~$W_n$. This algebra also can be interpreted as the Lie algebra of polynomial vector fields on the affine space $\mathbb{A}^n$.

The Lie algebra $W_n$ is simple. This result allows an essential generalization known as Siebert's theorem: the Lie algebra of polynomial vector fields on an affine variety $X$ is simple if and only if the variety $X$ is smooth~\cite{Sie}, see also~\cite{BF}.

The Lie algebra $W_n$ and its subalgebras have been studied in many works. In~\cite{Ru} subalgebras of minimal codimension in $W_n$ are described. It is also observed there that the automorphism group of the Lie algebra $W_n$ is naturally isomorphic to the automorphism group of the affine space~$\mathbb{A}^n$. Detailed proofs of these results can be found in \cite{Ba-2,KR}. Note that the famous Jacobian Conjecture is equivalent to the fact that every injective endomorphism of the Lie algebra $W_n$ is an automorphism, see~\cite{Ku,KR}.

By Ado's theorem, every finite-dimensional Lie algebra is embeddable into the matrix Lie algebra $\mathfrak{gl}_n$ for some positive integer $n$. It follows that every finite-dimensional Lie algebra is realized as a subalgebra in some Lie algebra $W_n$. Moreover, over an algebraically closed field and over real numbers it is possible to construct a so-called transitive embedding, that is an embedding such that the sections of the corresponding vector fields at the origin fill the entire tangent space, see~\cite{Go,Dr,Gr}.

In ~\cite{BL} the authors study polynomial Lie algebras. They include subalgebras in $W_n$ that are free $\mathbb{K}[x_1,\ldots,x_n]$-submodules of rank~$r$ with respect to the action of $ \mathbb{K}[x_1,\ldots,x_n]$ on $W_n$ by multiplications. In the paper~\cite{AMP}, finite-dimensional subalgebras in polynomial Lie algebras of rank $1$ are found.

The aim of this paper is to describe another natural class of finite-dimensional subalgebras of the Lie algebra $W_n$. Consider the lattice $\mathbb{Z}^n$ and the standard basis $e_1,\ldots e_n$ in~$\mathbb{Z}^n$. Recall that a $\mathbb{Z}^n$-grading on an algebra $A$ is a decomposition of $A$ into a direct sum of subspaces
\[
A=\bigoplus_{u\in\mathbb{Z}^n} A_u,
\]
where $A_vA_u\subseteq A_{v+u}$ for any $v,u\in\mathbb{Z}^n$. We call the subspace $A_u$ a \emph{homogeneous component} of degree $u$, an element $f\in A_u$ is a homogeneous element of degree $u$, and we use the notation $\deg(f)=u$. A subalgebra $B\subseteq A$ is called \emph{homogeneous} if $B=\oplus_{u\in\mathbb{Z}^n} (A_u\cap B)$ or, equivalently, $B$ is generated by homogeneous elements.

The algebra $\mathbb{K}[x_1,\ldots,x_n]$ has a natural $\mathbb{Z}^n$-grading given by the condition $\deg(x_i)=e_i$. We call this grading \emph{fine}. With respect to this grading, 
nonzero homogeneous components are exactly those components whose degree coordinates are non-negative. Such components are one-dimensional and spanned by monomials $x^a:=x_1^{a_1}\ldots x_n^{a_n}$. It follows that
\[
\deg(x_1^{a_1}\ldots x_n^{a_n})=(a_1,\ldots,a_n)\in\mathbb{Z}^n_{\ge 0}.
\]

Next we recall that a derivation $D$ of an algebra $A$ is {\it locally nilpotent} if for any $f\in A$ there exists a positive integer $s$ such that ${D^s(f)=0}$.
Let us assume that the derivation $D$ of the algebra $\mathbb{K}[x_1,\ldots,x_n]$ {\it is homogeneous}, i.e., $D$ maps $\mathbb{Z}^n$-homogeneous polynomials to $\mathbb{Z }^n$-homogeneous polynomials. As we already observed, for the fine $\mathbb{Z}^n$-grading homogeneous polynomials are precisely monomials.

By Leibniz's rule, every nonzero homogeneous derivation $D$ has well-defined {\it degree}, that is an element $e$ of the lattice $\mathbb{Z}^n$ such that for any $a\in\mathbb {Z}^n$ the operator $D$ maps a homogeneous polynomial of degree $a$ to a homogeneous polynomial of degree $a+e$.

It is easy to see that every homogeneous derivation $D$ of the algebra $\mathbb{K}[x_1,\ldots,x_n]$ is proportional to a derivation of exactly one of two types. These are homogeneous locally nilpotent derivations (type I):
$$
\nabla^a_i:= x_1^{a_1}\ldots x_{i-1}^{a_{i-1}}x_{i+1}^{a_{i+1}}\ldots x_n^{a_n}\partial_i
$$
and homogeneous derivations that are not locally nilpotent (type II):
$$
\Delta^p_{\beta}:=x_1^{p_1}\ldots x_n^{p_n} \sum_{j=1}^n \beta_jx_j\partial_j.
$$
In these notations, the parameter $a$ of the operator $\nabla^a_i$ is a vector in $\mathbb{Z}^n$ with all coordinates $a_s$ non-negative and the $i$th coordinate zero. We call $a$ {\it the vector of exponents} of the operator~$\nabla^a_i$. The parameter $p$ of the operator $\Delta^p_{\beta}$ is a vector in $\mathbb{Z}^n$ with non-negative coordinates $p_s$, and the parameter $\beta$ is a nonzero vector $(\beta_1,\ldots,\beta_n)$ in the space $\mathbb{K}^n$. Note that the degree of the homogeneous derivation $\nabla^a_i$ is $a-e_i$, and the degree of the homogeneous derivation $\Delta^p_{\beta}$ is $p$. 

We denote by $W^a_i$ the line $\langle\nabla^a_i\rangle$ and by $W^p$ the linear span of operators $\Delta^p_{\beta}$ for all possible~$\beta$. 
There is a decomposition
$$
W_n=\bigoplus_{a,i} W^a_i \, \oplus \, \bigoplus_p W^p.
$$ 
It is easy to see that this is a $\mathbb{Z}^n$-grading on the Lie algebra $W_n$ such that nonzero homogeneous components have degrees with at most one negative coordinate, and this negative coordinate is $-1$. Moreover, the components $W^a_i$ are one-dimensional, and the components $W^p$ have dimension~$n$. 

Along with the degree of a homogeneous derivation, it is useful to consider its weight. For derivations of type~I, we define the {\it weight} as 
$\omega(\nabla^a_i):=\sum_j a_j-1$, and for type~II we set $\omega(\Delta^p_{\beta}):=\sum_j p_j$. Thus, homogeneous derivations of weight $-1$ are exactly partial derivatives $\partial_i=\nabla^0_i$. Next, derivations of zero weight are $x_j\partial_i=\nabla^{e_j}_i$, $i\ne j$, and $\sum_j\beta_jx_j\partial_j=\Delta^0_{\beta}$. All other homogeneous derivations have positive weights.

It is easy to see that the weight defines a $\mathbb{Z}$-grading on the Lie algebra $W_n$. This is a factor-grading of the fine $\mathbb{Z}^n$-grading with respect to the projection
$\mathbb{Z}^n\to\mathbb{Z}$ mapping a lattice vector to the sum of its coordinates. 

In Section~\ref{s1-1} we consider derivations of the ring of Laurent polynomials which are homogeneous with respect to the fine grading on this ring. This provides a unified description of derivations of types~I and~II and leads to a convenient formula for the Lie bracket of such derivations.

The main aim of this paper is to describe finite-dimensional subalgebras of the Lie algebra $W_n$ that are homogeneous with respect to the fine $\mathbb{Z}^n$-grading. Every such subalgebra is generated by a finite set of homogeneous derivations. So the problem reduces to obtaining a finite dimensionality criterion for the Lie algebra generated by a finite set of derivations of the form $\nabla^a_i$ and~$\Delta^p_{\beta}$.

A finite dimensionality criterion for the Lie algebra generated by a finite set of derivations of type~I is obtained in~\cite{ALS} and~\cite{AZ}. For convenience of the reader, in Section~\ref{s2} we give a short proof of this criterion, which is valid over any field of characteristic zero.  

In Sections~\ref{s3}-\ref{s3-1} we obtain a finite dimensionality criterion for the Lie algebra generated by a finite set of derivations of type~II. All such Lie algebras are solvable, and if we assume that all generators have positive weights, they are also nilpotent. Some structural results and examples of finite-dimensional Lie algebras generated by derivations of type~II are given in Section~\ref{s3-2}.  

It is worth noting that in each case the finite dimensionality criterion is related to an acyclicity condition of a finite directed graph that is built on a finite set of homogeneous derivations. 

In the last section we discuss possible applications of the results obtained above and formulate problems for further research.

\section{Derivations of the Laurent polynomial ring}\label{s1-1}

In this section, we briefly review some facts on derivations of the Laurent polynomial ring $\mathbb{K}[x_1,x_1^{-1},\ldots,x_n,x_n^{-1}].$ We also discuss the relation of such derivations to our problem, and obtain formulas for Lie brackets we frequently use later. 

Similar to the case of the ring of polynomials, it is easy to show that every derivation $D$ of the ring of Laurent polynomials is of the form
$$
D=f_1\partial_1+f_2\partial_2+\ldots+f_n\partial_n,
$$
where $\partial_i:=\frac{\partial}{\partial x_i}$ is the partial derivative and $f_1,f_2,\ldots,f_n$ are arbitrary Laurent polynomials. 

The Lie algebra of derivations of the Laurent polynomial algebra is called the {\it Witt algebra}. We denote this Lie algebra by $\mathcal{W}_n$. It is a simple infinite-dimensional Lie algebra. Every derivation of the algebra $\mathbb{K}[x_1,\ldots,x_n]$ extends uniquely to a derivation of the Laurent polynomial algebra. This gives an embedding of the Lie algebras $W_n\subseteq\mathcal{W}_n$. 

In~\cite{Ba-1} it is proved that the automorphism group of the Witt algebra $\mathcal{W}_n$ is isomorphic to the automorphism group of the Laurent polynomial algebra $\mathbb{K}[x_1, x_1^{-1},\ldots,x_n,x_n^{-1}]$, and this group is a semi-direct product of the group $\mathrm{GL}_n(\mathbb{Z})$ and the algebraic torus $T^n$. In~\cite{Ch} a classification of homogeneous subalgebras of the Lie algebra of differential operators (of all orders) on the algebra $\mathbb{K}[x_1,x_1^{-1}]$ is obtained. 

It is well known that a locally nilpotent derivation annihilates any invertible element; see ~\cite[Principle~1~(b)]{Fr}. Since the ring of Laurent polynomials is generated by invertible elements, we conclude that this ring admits no nonzero locally nilpotent derivation. 

The fine $\mathbb{Z}^n$-grading on the algebra $\mathbb{K}[x_1,\ldots,x_n]$ extends to a fine $\mathbb{Z}^n$-grading on the algebra of Laurent polynomials:
$$
\mathbb{K}[x_1,x_1^{-1}\ldots,x_n,x_n^{-1}]=\bigoplus_{(a_1,\ldots,a_n)\in\mathbb{Z}^n} \mathbb{K} x_1^{a_1}\ldots x_n^{a_n}. 
$$ 
Next, homogeneous derivations from the Lie algebra $W_n$ can be extended uniquely to homogeneous derivations from the Lie algebra $\mathcal{W}_n$. Each homogeneous derivation from the Lie algebra $\mathcal{W}_n$ has the form
$$
D^c_{\alpha}:=x^c\sum_{j=1}^n \alpha_jx_j\partial_j, \quad \text{where} \quad c\in\mathbb{Z}^n, \\ \alpha\in\mathbb{K}^n. 
$$
Let $\epsilon_1,\ldots,\epsilon_n$ be the standard basis in $\mathbb{K}^n$. Then the extension of the derivation $\nabla^a_i$ to the Laurent polynomial algebra is nothing but the derivation $D^{a-e_i}_{\epsilon_i}$. The corresponding extension of the derivation $\Delta^p_{\beta}$ is the derivation $D^p_{\beta}$.  

Consider the map $\mathbb{K}^n\times\mathbb{Z}^n\to\mathbb{K}$, $(\alpha,u)\mapsto\langle\alpha,u\rangle:=\sum_{i=1}^n \alpha_i u_i$. It is easy to see that
$$
(x^c\sum_i\alpha_ix_i\partial_i)(x^u)=\langle\alpha,u\rangle x^{c+u}.
$$

\begin{lemma} \label{luse}
For any $c,d\in\mathbb{Z}^n$ and $\alpha,\beta\in\mathbb{K}^n$, we have
$$
[D^c_{\alpha},D^d_{\beta}]=D^{c+d}_{\langle\alpha,d\rangle\beta-\langle\beta,c\rangle\alpha}.
$$
\end{lemma} 

\begin{proof}
It suffices to check that the left and right parts of the equality act the same way on the generators $x_i$. We have
$$
(D^c_{\alpha}D^d_{\beta}-D^d_{\beta}D^c_{\alpha})(x_i)=D^c_{\alpha}(x^d\beta_i x_i)-D^d_{\beta}(x^c\alpha_i x_i)=\beta_i D^c_{\alpha}(x^{d+e_i})-\alpha_i D^d_{\beta}(x^{c+e_i})=
$$
$$
=\beta_i\langle\alpha,d+e_i\rangle x^cx^{d+e_i}-\alpha_i\langle\beta,c+e_i\rangle x^dx^{c+e_i}=(\beta_i\langle\alpha,d\rangle-\alpha_i\langle\beta,c\rangle)x^{c+d}x_i=
$$
$$
=D^{c+d}_{\langle\alpha,d\rangle\beta-\langle\beta,c\rangle\alpha}(x_i).
$$
This completes the proof.
\end{proof}

\section{Derivations of type I}
\label{s2}

A finite dimensionality criterion for the Lie algebra $\mathfrak{g}(\mathbb{D})$ generated by a set 
$$
\mathbb{D}=\left\{\nabla^{a(1)}_{i(1)},\ldots,\nabla^{a(m)}_{i(m)}\right\}
$$ 
of homogeneous locally nilpotent derivations can be found in~\cite[Theorem~5.1]{ALS}, see also \cite[Proposition~4.8]{AZ}. The results of these papers are obtained in a broader geometric context, and the ground field $\mathbb{K}$ is additionally assumed to be algebraically closed. But is turns out that the proof of the criterion we are interested in works over any field of characteristic zero. We give here a short proof of the criterion that is a combination of ideas from~\cite{ALS} and~\cite{AZ}. 

We associate with the set $\mathbb{D}$ a finite directed graph $\Gamma(\mathbb{D})$ with vertices $1,\ldots, m$. A pair $(s,j)$ is an edge in $\Gamma(\mathbb{D})$ if and only if the $i(s)$-th coordinate of the vector $a(j)$ is positive.

\begin{theorem}
\label{t1}
The Lie algebra $\mathfrak{g}(\mathbb{D})$ generated by a finite set $\mathbb{D}$ of derivations of type I is finite-dimensional if and only if either there are no oriented cycles in the graph $\Gamma(\mathbb{D})$, or the vertices of the graph $\Gamma(\mathbb{D})$ contained in oriented cycles correspond to zero-weight derivations.
\end{theorem}

\begin{proof}
Note that vertices corresponding to derivations of negative weight cannot be contained in oriented cycles. Suppose the graph $\Gamma(\mathbb{D})$ has an oriented cycle containing a vertex corresponding to a derivation of positive weight. Let us prove that in this case the Lie algebra $\mathfrak{g}(\mathbb{D})$ is infinitely dimensional.

{\it Case 1.1.}\ Let us consider the case of a cycle of length two. Assume that it corresponds to the derivations $\nabla^{a(1)}_{i(1)}$ and $\nabla^{a(2)}_{i(2)}$. For the sake of clarity, let $i(1)=1$, $\omega(\nabla^{a(1)}_1)>0$ and $i(2)=2$. Set $$
a(1)_2=c\ge 1, \quad a(2)_1=d\ge 1. 
$$
It is easy to deduce from Lemma~\ref{luse} that
$$
(\mathrm{ad}\,\nabla^{a(1)}_1)^{d+1}(\nabla^{a(2)}_2)=-c(d+1)!\nabla^b_1, \quad \text{where}\quad b_2=c(d+1)-1\ge c,
$$
and
$$
\omega(\nabla^b_1)=(d+1)\omega(\nabla^{a(1)}_1)+\omega(\nabla^{a(2)}_2)>\omega(\nabla^ {a(1)}_1).
$$
Thus, we can proceed from considering the pair $\left(\nabla^{a(1)}_1, \nabla^{a(2)}_2\right)$ to the pair $\left(\nabla^b_1, \nabla^{a(2)}_2\right)$. Continuing this process, we obtain infinitely many linearly independent commutators. This contradicts the finite dimensionality of the Lie algebra $\mathfrak{g}(\mathbb{D})$.

\smallskip

{\it Case 1.2.}\ Let $\mathbb{X}=\left\{\nabla^{a(1)}_{i(1)},\ldots,\nabla^{a(l)}_ {i(l)}\right\}$ be a sequence of derivations from the set $\mathbb{D}$ whose sum of weights is positive and which corresponds to vertices of an oriented cycle of minimal length among cycles with a positive sum of weights. Let us assume that $l\ge 3$.

Let $i(l-2)=s, i(l-1)=i$ and $i(l)=j$. Then
$$
a(l)_i:=c>0, \ a(l-1)_s>0, \ a(1)_j>0 \quad \text{and} \quad a(l-2)_i=a(l -1)_j=0.
$$
In particular, we have $i\ne s\ne j$. We obtain
$$
\left[\nabla^{a(l-1)}_{i(l-1)},\nabla^{a(l)}_{i(l)}\right]=c\nabla^b_j, \quad \text{where} \quad b_i=c-1 \quad \text{and} \quad b_s=a(l-1)_s+a(l)_s>0.
$$
Let us proceed from the set $\mathbb{D}$ to the set $\mathbb{D}'=\mathbb{D}\cup\{\nabla^b_j\}$. This does not change the Lie algebra $\mathfrak{g}(\mathbb{D})$ generated by the set.

Replacing $\mathbb{X}=\left\{\nabla^{a(1)}_{i(1)},\ldots,\nabla^{a(l)}_{i(l)}\right \}$ by $\mathbb{X}'=\left\{\nabla^{a(1)}_{i(1)},\ldots, \nabla^{a(l-2)}_{i (l-2)},\nabla^b_j\right\}$, we get a shorter cycle whose sum of weights is the same as that of the elements of~$\mathbb{X}$. In a finite number of steps we arrive at Case~1.1 and obtain a contradiction with the assumption that the Lie algebra~$\mathfrak{g}(\mathbb{D})$ is finite dimensional.

\smallskip

Now we come to the proof of the inverse implication in Theorem~\ref{t1}.

\smallskip

{\it Case 2.1.}\ Suppose that the graph $\Gamma(\mathbb{D})$ has no oriented cycle. Let us show that in this case the Lie algebra $\mathfrak{g}(\mathbb{D})$ is finite-dimensional and nilpotent.

Consider a graph $\mathcal{T}(\mathbb{D})$ whose vertices are numbered from $1$ to~$n$, and $(s,j)$ is an edge in $\mathcal{T}(\mathbb{D})$ if and only if there is a derivation $\nabla^a_j$ in the set $\mathbb {D}$ with $a_s>0$. With each oriented cycle $\mathcal{C}$ in the graph $\mathcal{T}(\mathbb{D})$ one associates an oriented cycle $\mathcal{C}'$ in the graph $\Gamma(\mathbb{D})$. Namely, let us choose an edge in the cycle $\mathcal{C}$, say $(s,j)$, and consider the corresponding pair of derivations $(\nabla^b_s,\nabla^a_j)$ with $a_s>0$. Let us associate with the next edge of the cycle $\mathcal{C}$ a pair with the same condition such that the first element of the pair is the derivation $\nabla^a_j$, and so on. Due to the finiteness of the set $\mathbb{D}$, we  eventually reach the derivation that was encountered earlier. Thus we obtain a oriented cycle $\mathcal{C}'$ in the graph $\Gamma(\mathbb{ D})$. Since by assumption there is no oriented cycle in the graph $\Gamma(\mathbb{D})$, there is no such cycle in the graph $\mathcal{T}(\mathbb{D})$ as well. Let us renumber the variables $x_1,\ldots, x_n$ so that for each edge $(s,j)$ in the graph $\mathcal{T}(\mathbb{D})$ we have $j>s$.

Recall that a derivation $D$ of the algebra $\mathbb{K}[x_1,\ldots,x_n]$ is \emph{triangular} if $D(x_k)\in\mathbb{K}[x_1,\ldots,x_ {k-1}]$ for all $2\le k\le n$ and $D(x_1)\in\mathbb{K}$.
In our situation, all derivations from $\mathbb{D}$ are triangular.

The statement we need follows from a more general result proved, for example, in \cite[Corollary~2]{ASh} and~\cite[Proposition~15.2.5]{FK}.

\begin{proposition}\label{inner_prop}
A finite set of triangular derivations generates a finite-dimensional nilpotent Lie algebra.
\end{proposition}

For convenience of the reader, we give a short proof of this statement in the case we are interested in, of a finite set of homogeneous triangular derivations.

\begin{proof}[Proof of Proposition \ref{inner_prop}]
Let $\nabla^a_i$ and $\nabla^b_j$ be homogeneous triangular derivations with $j>i$. Then
$$
[\nabla^a_i, \nabla^b_j]=b_i\nabla^f_j \quad \text{and} \quad f_i=b_i-1.
$$
Taking into account the upper triangularity, we conclude that if such a commutator is nonzero then its vector of exponents viewed from right to left is lexicographically smaller than the vector of exponents $b$. Thus, when commuting two homogeneous triangular derivations, we obtain either zero or a homogeneous triangular derivation with lower index equal to the maximum of the lower indices of the commutated elements, and the vector of exponents of the commutator is lexicographically smaller than the vector of exponents of the commutated element with the same low index. This shows that any sequence of multiple commutators of elements of our finite set vanishes in a finite number of steps. This implies that the corresponding Lie algebra is finite dimensional, and Engel's theorem shows that it is nilpotent. This completes the proof.
\end{proof}

{\it Case 2.2.}\ Let $\mathbb{D}=\mathbb{D}'\cup\mathbb{D}''$, where $\mathbb{D}'$ are those derivations for which the corresponding vertices of the graph $\Gamma(\mathbb{D})$ are contained in oriented cycles, and $\mathbb{D}''$~ are all other derivations. Let us assume that all derivations in $\mathbb{D}'$ have zero weight and show that the Lie algebra $\mathfrak{g}(\mathbb{D})$ is finite dimensional.

\begin{lemma} \label{ld}
Let $D_1\in\mathbb{D}'$, $D_2\in\mathbb{D}''$ and $D:=[D_1,D_2]$. Then either $D=0$ or $D$ is a homogeneous derivation of type~I and $D$ is not contained in any oriented cycle for the set $\mathbb{D}\cup\{ D\}$.
\end{lemma}

\begin{proof}
Let $D\ne 0$. Since the derivations $D_1$ and $D_2$ do not form a cycle of length two in the graph $\Gamma(\mathbb{D})$, there are exactly two cases left up to renumbering of variables.

{\it Case 1.}\ Let $D_1=x_1\partial_2$ and $D_2=\nabla^a_1$ with $a_2=0$. Then $D=-\nabla^a_2$. If $D$ is contained in some cycle, then replacing $D$ in this cycle by the pair $(D_2, D_1)$ gives a cycle in $\mathbb{D}$ that contains $D_2$. This leads to a contradiction.

{\it Case 2.}\ Let $D_1=x_1\partial_2$ and $D_2=\nabla^a_3$ with $a_2=c>0$. Then $D=c\nabla^b_3$, where $b_1=a_1+1$, $b_2=c-1$ and $b_s=a_s$ for $s>3$. If $D$ is contained in a cycle, then replacing $D$ in this cycle either by $D_2$ or by the pair $(D_1, D_2)$ (the latter may only be necessary if $a_1=0$) we obtain a cycle in $ \mathbb{D}$ that contains $D_2$. This is a contradiction.
\end{proof}

Thus, by adding commutators of the form $[D_1,D_2]$, we only increase the subset $\mathbb{D}''$. Since $D_1$ has zero weight, the weight of $[D_1,D_2]$ is equal to the weight of $D_2$. It follows that after adding a finite number of such commutators, we obtain a subset of $\mathbb{D}'''$ that is invariant under taking commutators with elements of $\mathbb{D}''$. Then the finite-dimensional nilpotent Lie algebra $\mathfrak{g}(\mathbb{D}'''$) generated by the subset $\mathbb{D}'''$ is invariant under such a commutation. The Lie algebra $\mathfrak{g}(\mathbb{D}')$ is generated by derivations of zero weight and is therefore finite dimensional. Thus, the Lie algebra $\mathfrak{g}(\mathbb{D})$ is the semidirect product $\mathfrak{g}(\mathbb{D}') \rightthreetimes \mathfrak{g}(\mathbb{D}''')$ of two finite-dimensional Lie algebras. 

\smallskip

This completes the proof of Theorem~\ref{t1}.
\end{proof}

\begin{remark} \label{r1}
Let the Lie algebra $\mathfrak{g}(\mathbb{D})$ be finite dimensional. It is easy to show that subsets of derivations from $\mathbb{D}'$ corresponding to maximal oriented cycles give rise to the special linear Lie algebras $\sl(r_i)$; see ~\cite[Proposition~4.4]{ALS}. The Lie algebra $\mathfrak{g}(\mathbb{D}')$ is a direct sum of such subalgebras. In other words, it is a Lie algebra of type~A. The Lie algebra $\mathfrak{g}(\mathbb{D}')$ is a semisimple part of the Lie algebra~$\mathfrak{g}(\mathbb{D})$. Finally, the algebra $\mathfrak{g}(\mathbb{D}''')$ is the nilpotent radical of the Lie algebra~$\mathfrak{g}(\mathbb{D})$. 
\end{remark}

It should be noted that the nilpotent radical $\mathfrak{g}(\mathbb{D}''')$ is a linear span of homogeneous locally nilpotent derivations. Consequently, it does not contain derivations of type~II. In contrast, derivations of type~II form bases in Cartan subalgebras of the Lie algebras $\mathfrak{sl}(r_i)$. In particular, derivations of type~II generate a commutative Lie subalgebra in $\mathfrak{g}(\mathbb{D})$. 

\section{Derivations of type II}
\label{s3}
Now we consider the problem of finite dimensionality of the Lie algebra $\mathfrak{g}(\mathcal{D})$ generated by a finite set $\mathcal{D}$ of homogeneous derivations 
$\Delta^{p(1)}_{\beta(1)},\ldots,\Delta^{p(k)}_{\beta(k)}$. 

Let $\beta\in\mathbb{K}^n$ and $p\in\mathbb{Z}^n$. Recall that $\langle\beta,p\rangle:=\beta_1p_1+\ldots+\beta_np_n$.  Let us reformulate Lemma~\ref{luse} to apply it to our situation.

\begin{lemma} \label{l1}
For any two homogeneous derivations of type II we have
$$
[\Delta^p_{\beta},\Delta^q_{\gamma}]=\Delta^{p+q}_{\langle\beta,q\rangle\gamma-\langle\gamma,p\rangle\beta}.
$$
\end{lemma}

Consider the case of derivations with the same lower indices. 

\begin{lemma} \label{l2}
Let $p,q\in\mathbb{Z}^n_{\ge 0}$ be nonzero vectors and $\beta\in\mathbb{K}^n$. Then the following conditions are equivalent:
\begin{enumerate}
\item[{\rm (i)}]
the derivations $\Delta^p_{\beta}$ and $\Delta^q_{\beta}$ generate a finite dimensional Lie algebra; 
\item[{\rm (ii)}]
$[\Delta^p_{\beta}, \Delta^q_{\beta}]=0$;
\item[{\rm (iii)}] 
$\langle\beta,p\rangle=\langle\beta,q\rangle$. 
\end{enumerate}
\end{lemma}

\begin{proof}
It is evident that condition (ii) implies (i). Equivalence of ${\rm (ii)}$ and ${\rm (iii)}$ follows from Lemma~\ref{l1}. 

Let us prove implication ${\rm (i)}\Rightarrow {\rm (iii)}$. By computing multiple commutators of the derivations $\Delta^p_{\beta}$ and $\Delta^q_{\beta}$, we obtain derivations of degrees $c'p+c''q$, $c',c''\in\mathbb{Z}_{\ge 0}$, whose weights increase. Replacing a derivation with a proportional one, we can assume that the lower index remains equal $\beta$. It follows from the assumption of finite dimensionality that at some moment we come to a derivation $\Delta^{c_1p+c_2q}_{\beta}$ with
$$
[\Delta^{c_1p+c_2q}_{\beta},\Delta^p_{\beta}]=[\Delta^{c_1p+c_2q}_{\beta},\Delta^q_{\beta}]=0.
$$

These conditions imply $\langle\beta,c_1p+c_2q-p\rangle=\langle\beta,c_1p+c_2q-q\rangle=0$, hence condition ${\rm (iii)}$ holds.
\end{proof}

Now we proceed to examine the case of derivations with different lower indices.

\begin{lemma} \label{l3}
Let $p,q\in\mathbb{Z}^n_{\ge 0}$ be nonzero vectors and the vectors $\beta,\gamma\in\mathbb{K}^n$ are not proportional. If the operators $\Delta^p_{\beta}$ and $\Delta^q_{\gamma}$ generate a finite-dimensional Lie algebra, then there exists $r\in\mathbb{Z}_{\ge 0}$ such that $\langle\beta,rp+q\rangle=0$. 
\end{lemma}

\begin{proof}
We have
\[
 [\underbrace{ \Delta^p_{\beta},\ldots, [\Delta^p_{\beta}}_{\mbox{$t$ times}},\Delta^q_{\gamma}] \ldots] = \Delta^{tp+q}_{u\beta+v\gamma} \ ,
 \]
where $u,v\in\mathbb{K}$. In accordance with the condition of finite dimensionality, this operator is identically zero for sufficiently large values of $t$. In particular, the coefficient $v$ is zero. Let $r$ be the largest value of the parameter $t$ such that the coefficient $v$ is not zero.

We obtain
$$
[\Delta^p_{\beta},\Delta^{rp+q}_{u\beta+v\gamma}]=\Delta^{(r+1)p+q}_{\langle\beta,rp+q\rangle(u\beta+v\gamma)-\langle u\beta+v\gamma,p\rangle\beta}.
$$

It follows that $\langle\beta,rp+q\rangle v =0$. Thus we have $\langle\beta,rp+q\rangle =0$, as claimed.
\end{proof}

\begin{lemma} \label{l4}
Under the conditions of Lemma~\ref{l3}, let $r$ and $s$ be the smallest non-negative integers such that $\langle\beta,rp+q\rangle=\langle\gamma,p+sq\rangle=0$. Then either $r=0$ or $s=0$. 
\end{lemma}

\begin{proof}
Suppose $r\ge 2$ and $s\ge 1$. Recall that 
$$
[\Delta^q_{\gamma}, \Delta^p_{\beta}]=\Delta^{p+q}_{\langle\gamma,p\rangle\beta-\langle\beta,q\rangle\gamma}.
$$
By assumption, the coefficients at $\beta$ and $\gamma$ in the operator on the right are nonzero. Commuting the operator $\Delta^p_{\beta}$ with $\Delta^q_{\gamma}$ and $\Delta^p_{\beta}$ successively, we obtain
\begin{equation} \label{f1}
[\Delta^q_{\gamma}, \Delta^{tp+(t-1)q}_{u\beta+v\gamma}]=\Delta^{tp+tq}_{\langle\gamma, tp+(t-1)q\rangle(u\beta+v\gamma)-\langle u\beta+v\gamma,q\rangle\gamma}
\end{equation}
and
\begin{equation} \label{f2}
[\Delta^p_{\beta}, \Delta^{tp+tq}_{u'\beta+v'\gamma}]=\Delta^{(t+1)p+tq}_{\langle\beta, tp+tq\rangle(u'\beta+v'\gamma)-\langle u'\beta+v'\gamma,p\rangle\beta}. 
\end{equation}

It follows from the finite dimensionality condition that at some moment these commutators should vanish. Consequently, the coefficients at $\beta$ and $\gamma$ will vanish. 

If we have $\langle\gamma, tp+(t-1)q\rangle u=0$, where $u\ne 0$ and $\langle\gamma,p+sq\rangle=0$, then $1\le s=\frac{t-1}{t}$, and we come to a contradiction. Thus, the last nonzero coefficient at $\beta$ appears in (\ref{f1}).

Now let $\langle\beta,tp+tq\rangle v'=0$ with $v'\ne 0$. Then $\langle\beta,p+q\rangle=0$, which contradicts the assumption $r\ge 2$. Thus, the last nonzero coefficient at $\gamma$ occurs in formula~(\ref{f2}). 

Let us suppose that the coefficient at $\beta$ has become the first to take only zero values. Then, in all of the following steps, the coefficient at $\gamma$ in formula~(\ref{f1}) is equal to
$$
v\langle\gamma,tp+(t-2)q\rangle.
$$ 

Hence, if this coefficient equals zero, we obtain a contradiction because of~${1\le s=\frac{t-2}{t}}$. 

Let us assume now that the coefficient at $\gamma$ has become the first to take only zero values. Then in all subsequent steps, the coefficient at $\beta$ in formula~(\ref{f2}) is equal to
$u'\langle\beta,(t-1)p+tq\rangle.$ If it is equal to zero, then $2\le r=\frac{t-1}{t}$ gives a contradiction. Therefore, neither of these cases is possible.

It remains to consider the case $r=s=1$. We have $\langle\beta,p+q\rangle=\langle\gamma,p+q\rangle=0$, but the pairings 
$$
\langle\beta,p\rangle, \quad \langle\beta,q\rangle, \quad \langle\gamma,p\rangle, \quad \text{and} \quad \langle\gamma,q\rangle 
$$
are nonzero. We obtain
$$
[\Delta^q_{\gamma},[\Delta^q_{\gamma}, \Delta^p_{\beta}]]=[\Delta^q_{\gamma},\Delta^{p+q}_{\langle\gamma,p\rangle\beta-\langle\beta,q\rangle\gamma}]= 
\Delta^{p+2q}_{\omega},
$$
where
$$
\omega=\langle\gamma,p+q\rangle(\langle\gamma,p\rangle\beta-\langle\beta,q\rangle\gamma)-\langle\langle\gamma,p\rangle\beta-\langle\beta,q\rangle\gamma,q\rangle\gamma=v\gamma
$$
with $v=-\langle\langle\gamma,p\rangle\beta-\langle\beta,q\rangle\gamma,q\rangle=\langle\beta,q\rangle\langle\gamma,q-p\rangle\ne 0.$ 

Thus, together with the elements $(\Delta^q_{\gamma}, \Delta^p_{\beta})$, the Lie algebra $\mathfrak{g}(\mathcal{D})$ also contains the elements $(\Delta^{p+2q}_{\gamma}, \Delta^p_{\beta})$ with the condition
$$
\langle\beta,p+2q+p\rangle=\langle\gamma,p+2q+p\rangle=0
$$ 
and $\langle\beta,p\rangle$, $\langle\beta,p+2q\rangle$, $\langle\gamma,p\rangle$, and $\langle\gamma,p+2q\rangle$ are nonzero. 

Repeating the procedure, we pass from this pair to the pair $(\Delta^{p+2(p+2q)}_{\gamma}, \Delta^p_{\beta})$, and so on. This is a contradiction with the finite dimensionality condition. This completes the proof.
\end{proof}

Let us turn to the finite set $\mathcal{D}=\left\{\Delta^{p(1)}_{\beta(1)},\ldots,\Delta^{p(k)}_{\beta(k)}\right\}$. Let $\mathfrak{g}(\mathcal{D})$ be the Lie algebra generated by this set.
Since $[\Delta^0_{\beta},\Delta^q_{\gamma}] = \Delta^q_{\langle\beta,q\rangle\gamma}$, i.e., the commutator is proportional to $\Delta^q_{\gamma}$, the property of the Lie algebra $\mathfrak{g}(\mathcal{D})$ to be finite-dimensional remains unaffected by the presence of zero-weight derivations.  Therefore, we further assume that all derivations in the set $\mathcal{D}$ have positive weights.

Let us construct a directed graph $\Gamma(\mathcal{D})$. Its vertices are labeled by $1$ to $k$ and there is an edge from vertex $i$  to vertex $j$ if and only if the vectors $\beta(i)$ and $\beta(j)$ are not proportional and~$\langle\beta(i),p(j)\rangle\ne 0$.

\begin{lemma}\label{l5}
If the Lie algebra $\mathfrak{g}(\mathcal{D})$ is finite dimensional, then the graph $\Gamma(\mathcal{D})$ does not contain oriented cycles. 
\end{lemma}

\begin{proof}
From Lemma~\ref{l4} it follows that the graph $\Gamma(\mathcal{D})$ does not contain oriented cycles of length $2$. Renumbering elements of the set $\mathcal{D}$, we may assume that derivations $\Delta^{p(1)}_{\beta(1)},\ldots,\Delta^{p(s)}_{\beta(s)}$ correspond to an oriented cycle of the minimal length $s\ge 3$. 

By the cyclicity condition, we have
$$
\langle\beta(1),p(2)\rangle\ne 0, \quad \langle\beta(2),p(3)\rangle\ne 0, \quad \ldots, \quad \langle\beta(s),p(1)\rangle\ne 0, 
$$ 
and by the minimality condition we conclude that $\langle\beta(i),p(j)\rangle=0$ for all other pairs of indices $1\le i\ne j\le s$. 

Let us consider
\[
 \left[\Delta^{p(1)}_{\beta(1)},\left[\Delta^{p(2)}_{\beta(2 )},\ldots \left[\Delta^{p(s-2)}_{\beta(s-2)},\Delta^{p(s-1)}_{\beta(s-1)}\right]\ldots\right]\right].
\]

It is easy to see that this commutator is proportional to the operator $\Delta^{p(1)+\ldots+p(s-1)}_{\beta(s-1)}$ with a nonzero coefficient.

By
$$
\langle\beta(s-1),p(s)\rangle\ne 0, \qquad \langle \beta(s), p(1)+\ldots+p(s-1)\rangle\ne 0,
$$
we obtain a contradiction to Lemma~\ref{l4} applied to the operators $\Delta^{p(1)+\ldots+p(s-1)}_{\beta(s-1)}$ and $\Delta^{p(s)}_{\beta(s)}$. This completes the proof.
\end{proof}

We are now in a position to present the main result.

\begin{theorem}
\label{t2}
Let $\mathcal{D}=\left\{\Delta^{p(1)}_{\beta(1)},\ldots,\Delta^{p(k)}_{\beta(k)}\right\}$ be a set of derivations of type~II and of positive weights. The Lie algebra $\mathfrak{g}(\mathcal{D})$ generated by this set is finite dimensional if and only if the elements of $\mathcal{D}$ can be renumbered in such a way that the following conditions hold:
\begin{enumerate}
\item[1)]
if $\beta(i)$, $\beta(j)$ are proportional, then $\langle\beta(i),p(i)-p(j)\rangle=0$;
\item[2)] 
if $\beta(i)$, $\beta(j)$ are not proportional and $j>i$, then $\langle\beta(j),p(i)\rangle=0$ and there exists $r_{ij}\in\mathbb{Z}_{\ge 0}$ such that $\langle\beta(i),p(j)+r_{ij}p(i)\rangle=0$.
\end{enumerate}
\end{theorem} 

\section{Proof of the main result}
\label{s3-1}

We come to the proof of Theorem~\ref{t2}. The necessity of conditions~1) and~2) follow from Lemmas~\ref{l2}-\ref{l3} and Lemma~\ref{l5}. Now we prove the sufficiency of these conditions. 

Let the derivations be numbered in the manner specified in Theorem~\ref{t2}. Replacing the derivations by proportional ones, we assume additionally that if the vectors $\beta(i)$ and $\beta(j)$ are proportional, then they are equal. 

Furthermore, it can be assumed that derivations with the same $\beta$ are consecutive in the set~$\mathcal{D}$. If this is not the case, a derivation with the given $\beta$ and the smallest number can be rearranged immediately before the next derivation with the same $\beta$. It follows from conditions~1) and~2) that such a rearrangement does not break the conditons on  the ordering, and by such permutations we obtain the required one.

Let us decompose derivations from the set $\mathcal{D}$ into classes $C_1,\ldots,C_s$ according to the value of the vector $\beta$. We aim to show that, repeatedly commuting these derivations, we can obtain, up to proportionality, only a finite number of new homogeneous derivations.

We have
\[
 \left[\Delta^{p(i)}_{\beta(i)},\Delta^{p(j)}_{\beta(j)}\right]  = \langle\beta(i),p(j)\rangle\Delta^{p(i)+p(j)}_{\beta(j)}, \qquad j>i.
\]

It follows that adding a commutator to the set $\mathcal{D}$ does not affect the number of classes, and a nonzero commutator of two derivations belongs to a class with bigger index.

Any two derivations belonging to the same class commute. In particular, the first class cannot be extended by adding commutators.

Let us prove by induction on $l$ that the class $C_l$, $l \le s$, can be extended only by a finite number of elements. 

By induction hypothesis, we can assume that all possible multiple commutators have already been added to the classes $C_1,\ldots,C_{l-1}$. After this the number of elements in all classes remains finite and the aforementioned conditions on derivations are satisfied. 

If an infinite number of elements can be added to the class $C_l$, then there exists a derivation $\Delta^{p(j)}_{\beta(j)}\in C_l$ and an infinite set of derivations $\{\Delta^{q(t)}_{\beta(j)}\}_{t\in\mathbb{Z}_{>0}}$, where $\Delta^{q(1)}_{\beta(j)}=\Delta^{p(j)}_{\beta(j)}$, such that each subsequent element $\Delta^{q(t)}_{\beta(j)}$ is obtained by commuting $\Delta^{q(t-1)}_{\beta(j)}$ with some derivation $\Delta^{p(b)}_{\beta(b)}$ from the previous classes. In particular, we have $q(t)=q(t-1)+p(b)$. 

Let us verify that when we add a derivation $\Delta^{q(t)}_{\beta(j)}$ to the set, the conditions for the new set still hold. It is clear that $\langle\beta(a),q(t)\rangle=0$, $a>j$, and  $\langle\beta(j),q(t)-q(t-1)\rangle=\langle\beta(j),p(b)\rangle=0$. Hence, the same equality holds for all derivations from this class. 

We can also assume that for any derivation $\Delta^{p(i)}_{\beta(i)}$ from the previous classes, the condition $\langle\beta(i),q(t-1)+rp(i)\rangle=0$ is satisfied for some $r\in\mathbb{Z}_{\ge 0}$ that depends on the number $i$.

Let us prove that 
$$
\langle\beta(i),q(t)+r^*p(i)\rangle=0
$$ 
for a suitable $r^*\in\mathbb{Z}_{\ge 0}$. Let $\Delta^{p(i)}_{\beta(i)} \in C_u$, $\Delta^{p(b)}_{\beta(b)} \in C_v.$ If $u<v$ then by induction hypothesis there is $r'\in\mathbb{Z}_{\ge 0}$ such that 
$$
\langle\beta(i),p(b)+r'p(i)\rangle=0.
$$ 
So we have
$$
\langle\beta(i),q(t)+(r+r')p(i)\rangle=\langle\beta(i),q(t-1)+rp(i)+p(b)+r'p(i)\rangle=0.
$$

If $u>v$ then $\langle\beta(i),p(b)\rangle=0$ and 
$$
\langle\beta(i),q(t)+rp(i)\rangle=\langle\beta(i),q(t-1)+rp(i)\rangle=0.
$$

If $u=v$, that is $\beta(i)=\beta(b)$, we have to consider two cases. 

\smallskip

1) If $r=0$ then $[\Delta^{q(t-1)}_{\beta(j)},\Delta^{p(b)}_{\beta(b)}]=0$, because of $\langle\beta(b),q(t-1)\rangle=\langle\beta(i),q(t-1)\rangle=0$. Hence, this case is not realized. 

\smallskip

2) Let $r>0$. We have $\langle\beta(i),p(i)\rangle=\langle\beta(i),p(b)\rangle$, since $\Delta^{p(i)}_{\beta(i)}$, $\Delta^{p(b)}_{\beta(b)}$ belong to the same class. 
This implies
$$
\langle\beta(i),q(t)+(r-1)p(i)\rangle=\langle\beta(i),q(t)-p(b)+rp(i)\rangle=\langle\beta(i),q(t-1)+rp(i)\rangle=0.  
$$
Thus, we have verified that all conditions hold. 

With a derivation $\Delta^{q(t)}_{\beta(j)}$ one associates a vector $R_t=(r_1,\ldots,r_m)$ with non-negative coordinates, where $m$ is the total number of derivations in classes $C_1,\ldots, C_{l-1}.$ Namely, the number $r_i$ is the smallest non-negative integer such that $\langle\beta(i),q(t)+r_ip(i)\rangle=0$, where $\Delta^{p(i)}_{\beta(i)}$ successively runs over all derivations of classes $C_1,\ldots, C_{l-1}.$

It follows from the previous discussion that $R_t$ is strictly less than $R_{t-1}$ lexicographically, where coordinates of vectors are viewed from right to left. It is well known that lexicographically every non-empty subset contains a minimal element. On the other hand, we obtain an infinite sequence of derivations $\Delta^{q(t)}_{\beta(j)}$. This contradicts to well-ordering of the vectors $R_t$. This completes the proof.

{\section{Structure of the Lie algebra $\mathfrak{g}(\mathcal{D})$}
\label{s3-2}

The Lie algebra $\mathfrak{g}(\mathcal{D})$ in Theorem~\ref{t2} is positively graded. So if $\mathfrak{g}(\mathcal{D})$ is finite dimensional then it is nilpotent. If one adds derivations of weight zero to the set of generators $\mathcal{D}$, then finite dimensionality implies solvability but nilpotency is lost. 

Note that since a derivation of type~II has all degree coordinates non-negative, the Lie algebra generated by such derivations does not contain derivations of type~I. 

\begin{lemma}
\label{l36}
For any $d\ge 1$ we have
$$
(\mathrm{ad}\,\Delta^q_{\gamma})^d(\Delta^p_{\beta})=s_d\Delta^{p+dq}_{\omega_d}, 
$$
where $s_1=1$, $s_d=\prod_{i=0}^{d-2}\langle\gamma,p+iq\rangle$, $d>1$, and $\omega_d=\langle\gamma,p+(d-1)q\rangle\beta-d\langle\beta,q\rangle\gamma$. 
\end{lemma} 
\begin{proof}
The proof is by induction on $d$.
\end{proof}

Lemma~\ref{l36} provides a description of the subalgebra generated by two elements in $\mathcal{D}$ for the finite dimensional Lie algebra~$\mathfrak{g}(\mathcal{D})$ from Theorem~\ref{t2}. 

\begin{proposition}
Let $\langle\beta, q\rangle=0$ and $r\in\mathbb{Z}_{\ge 0}$ be the smallest number such that $\langle\gamma, p+rq\rangle=0$. Then the Lie algebra $\mathfrak{g}(\Delta^q_{\gamma}, \Delta^p_{\beta})$ has dimension $r+2$ and it is $(r+1)$-step nilpotent. Moreover, the derivations 
$$
\Delta^q_{\gamma}, \Delta^p_{\beta}, \Delta^{p+q}_{\beta},\ldots,\Delta^{p+rq}_{\beta}
$$
form a basis of this Lie algebra. 
\end{proposition} 

\begin{proof}
It follows immediately from Lemma~\ref{l36} and Theorem~\ref{t2}. 
\end{proof}

Elements of this basis can be multiplied by suitable scalars such that for the new basis $X_1,X_2,\ldots,X_{r+2}$ we have
$[X_1,X_i]=X_{i+1}, 1<i<r+2$, and all other commutators are zero. Such a Lie algebra is called a \emph{model filiform} Lie algebra; see e.g.~\cite{GKh}. 

\begin{example}
Let $D_1=\Delta^q_{\gamma}=x_1^2\partial_1-x_1x_2\partial_2$ and $D_2=\Delta^p_{\beta}=x_2^2\partial_2$. Here 
$$
q=(1,0),\quad \gamma=(1,-1),\quad p=(0,1),\quad \beta=(0,1).
$$ 
It follows that $\langle\beta, q\rangle=0$, $r=1$, and a basis of the algebra $\mathfrak{g}(D_1,D_2)$ is $\{D_1,D_2,D_3\}$, where $D_3=[D_1,D_2]=-\Delta^{p+q}_{\beta}=-x_1x_2^2\partial_2$. 
\end{example}

\begin{example}
Let $D_1=\Delta^q_{\gamma}=x_1^2\partial_1-2x_1x_2\partial_2$ and $D_2=\Delta^p_{\beta}=x_2^2\partial_2$. Here 
$$
q=(1,0),\quad \gamma=(1,-2),\quad p=(0,1),\quad \beta=(0,1).
$$ 
It follows that $\langle\beta, q\rangle=0$, $r=2$, and a basis of the algebra $\mathfrak{g}(D_1,D_2)$ is $\{D_1,D_2,D_3,D_4\}$, where 
$$
D_3=[D_1,D_2]=-2\Delta^{p+q}_{\beta}=-2x_1x_2^2\partial_2 \quad \text{and} \quad
D_4=[D_1,D_3]=2\Delta^{p+2q}_{\beta}=2x_1^2x_2^2\partial_2.
$$
\end{example}

\begin{example}
Let $m,l \in \mathbb{Z}_{>0}$. Set
$$
D_1=\Delta^q_{\gamma}=x_1x_2(2x_1\partial_1-x_2\partial_2)\quad \text{and}\quad D_2=\Delta^p_{\beta}=x_1^mx_2^{2m+l}(x_1\partial_1-x_2\partial_2),
$$ 
where 
$$
q=(1,1),\quad \gamma=(2,-1),\quad p=(m,2m+l),\quad \beta=(1,-1).
$$ 
Hence $\langle\beta, q\rangle=0$ and $r=l$. It follows that the Lie algebra $\mathfrak{g}(D_1,D_2)$ has dimension $l+2$ and is $(l+1)$-step nilpotent.  
\end{example}

In our opinion, Theorem~\ref{t2} describes a natural class of finite-dimensional graded nilpotent Lie algebras which can be defined in terms of multiplicities of occurrences of each vector $\beta(i)$, values of $\langle\beta(i),p(i)\rangle$ and the set of non-negative integers $\{r_{ij}\}$. It would be interesting to develop a structure theory of such Lie algebras.

\section{Concluding remarks and open questions}\label{s4}

The problem of obtaining a finite dimensionality criterion for a Lie algebra generated by a finite set of homogeneous derivations of types~I and~II seems to be important. 

\begin{problem}
Find a criterion of finite dimensionality of a Lie algebra generated by a finite set of homogeneous derivations of the ring $\mathbb{K}[x_1,\ldots,x_n]$ and describe the structure of the arising finite-dimensional Lie algebras. 
\end{problem}

Theorem~\ref{t1} and Theorem~\ref{t2} give necessary conditions of finite dimensionality for a set of derivations of types~I and~II. Suitable algebraic and combinatorial concepts have yet to be developed to obtain a necessary and sufficient condition. It is also natural to ask the more general question of describing the finite-dimensional homogeneous subalgebras of the Witt algebra $\mathcal{W}_n$. 

\begin{problem}
Find a criterion for finite dimensionality of the Lie algebra generated by a finite set of homogeneous derivations of the ring of Laurent polynomials $\mathbb{K}[x_1,x_1^{-1},\ldots,x_n,x_n^{-1}]$ and describe the structure of the arising finite-dimensional Lie algebras.
\end{problem}

It would be interesting to describe finite-dimensional Lie algebras generated by homogeneous derivations for other graded algebras. We say that a grading is \emph{fine} if all its homogeneous components are at most one-dimensional. It is well known that an integrally closed affine algebra with a fine grading is exactly the semigroup algebra $\mathbb{K}[S(\sigma)]$ of the semigroup of integer points $S(\sigma)$ in a polyhedral cone $\sigma$ in the vector space $\mathbb{Q}^n$ over the field of rational numbers. In other words, this is the algebra of regular functions on a normal affine toric variety. 

Homogeneous locally nilpotent derivations of such algebras are given by remarkable combinatorial objects called Demazure roots. The concept of a Demazure root goes back to~\cite{De}; in the present context it is invented in~\cite{Li1,Li2}. One of the aims of this paper is to develop an approach to describing homogeneous derivations of polynomial algebras that are not locally nilpotent, which is analogous to the description of homogeneous locally nilpotent derivations in terms of Demazure roots. The above proposed description of derivations of type~II provides such an approach. In the future, we plan to generalize Theorem~\ref{t2} to the case of semigroup algebras of affine semigroups with fine grading. In the case of Theorem~\ref{t1}, this has been done in~\cite{ALS, AZ}.  

\begin{problem}
Let $\mathbb{K}[S(\sigma)]$ be the semigroup algebra of the semigroup $S(\sigma)$ of integer points in a polyhedral cone $\sigma$ in the vector space $\mathbb{Q}^n$ over the field of rational numbers. 
Find a criterion for finite dimensionality of the Lie algebra generated by a finite set of homogeneous derivations of the ring $\mathbb{K}[S(\sigma)]$ with respect to the fine grading and describe the structure of the arising finite-dimensional Lie algebras.  
\end{problem}

An essential motivation for solving the above problems over an algebraically closed field of characteristic zero is the following remarkable result obtained recently by Hanspeter Kraft and
Mikhail Zaidenberg. Let $X$ be an irreducible affine algebraic variety and $G$ be a subgroup in $\mathrm{Aut}(X)$ generated as an abstract group by a family of connected algebraic subgroups $G_i$. Then~\cite[Theorem~A]{KZ} claims that $G$ is an algebraic group if and only if the images of the tangent algebras $\mathrm{Lie}(G_i)$ to the groups $G_i$ in the Lie algebra of all polynomial vector fields on the variety $X$ generate a finite-dimensional Lie subalgebra~$L(G)$. Moreover, in this situation the Lie algebra $L(G)$ is the tangent algebra to the algebraic group $G$.

We say that a linear algebraic group $G$ is \emph{special} if it is generated by one-parameter subgroups isomorphic to the additive group $(\mathbb{K},+)$ of the ground field. As shown in~\cite[Lemma~1.1]{Po}, a connected linear algebraic group $G$ is special if and only if $G$ has no nontrivial character. The same condition is equivalent to the fact that a maximal reductive subgroup of $G$ is semisimple, or that the solvable radical of $G$ is unipotent. 

It is well known there is one-to-one correspondence between one-parameter additive subgroups of the group $\mathrm{Aut}(X)$ and locally nilpotent derivations on the algebra of regular functions $\mathbb{K}[X]$; cf.~\cite[Section~1.5]{Fr}. We conclude that Theorem~\ref{t1} gives a complete description of the special subgroups of the group $\mathrm{Aut}(\mathbb{A}^n)$ which are normalized by the maximal torus of the group $\mathrm{Aut}(\mathbb{A}^n)$.  In particular, Theorem~\ref{t1} and Remark~\ref{r1} imply the following result. 

\begin{proposition}
Assume that the ground field $\mathbb{K}$ is an algebraically closed field of characteristic zero. If an action of a special algebraic group $G$ on the affine space $\mathbb{A}^n$ is normalized by the maximal torus of non-degenerate diagonal matrices, then a semisimple part of the group $G$ coincides with the semisimple part of a Levi subgroup of the group $\text{SL}_n$ with respect to the standard action of $\text{SL}_n$ on $\mathbb{A}^n$.
\end{proposition}

\begin{proof}
It suffices to notice that homogeneous locally nilpotent derivations of zero weight are precisely $x_i\partial_j$ with $i\ne j$. These derivations give rise to one-dimensional subgroups $E+cE_{ij}$, $c\in\mathbb{K}$, of the group $\text{SL}_n$, where $E$ is the unit matrix and $E_{ij}$ are the matrix units. Finally, such derivations correspond to vertices of oriented cycles of the graph $\Gamma(\mathbb{D})$  if and only if the corresponding subgroups generate the semisimple part of some Levi subgroup of the group $\text{SL}_n$;
see~\cite[Theorem~30.1]{Hum}. 
\end{proof}

It would be interesting to describe finite-dimensional subalgebras generated by locally nilpotent derivations in Lie algebras of derivations not only for the ring of polynomials but also for another affine algebras.

Finally, it is of special interest to study the structural properties of infinite-dimensional subalgebras of the Lie algebras $W_n$ and $\mathcal{W}_n$ generated by a finite set of homogeneous derivations. We plan to develop the corresponding structure theory in subsequent publications. 

\bigskip

\emph{Acknowledgements.} The authors are grateful to the anonymous referee for a careful reading of the text and valuable suggestions. Thanks are also due to Andriy Regeta for useful comments on the results cited in the article.

\end{document}